\documentclass[10pt]{article}

\usepackage{amsmath}
\usepackage{amssymb}
\usepackage{amsfonts}
\usepackage{amsthm}
\usepackage{latexsym}

\newtheorem{theorem}{Theorem}[section]
\newtheorem{lem}[theorem]{Lemma} 
\newtheorem{lemma}[theorem]{Lemma} 
\newtheorem{observation}[theorem]{Observation}

\newtheorem{prop}[theorem]{Proposition}

\newtheorem{claimA}{Claim \textrm{A}\hspace{-1.4mm}}
\newtheorem{claimB}{Claim \textrm{B}\hspace{-1.4mm}}

\newtheorem{claim}{Claim}

\theoremstyle{definition}

\theoremstyle{remark}

\newtheorem{cProofA}{Proof of Claim {\rm A}\hspace*{-1.4mm}} 
\newtheorem{cProofB}{Proof of Claim {\rm B}\hspace*{-1.4mm}} 

\newtheorem{cProof}{Proof of Claim}

\begin{document}

\title{Equimatchable factor-critical graphs and independence number 2}

\author{
Eduard Eiben\thanks{Institute of Computer Graphics and Algorithms (186), Vienna University of Technology (TU Wien), Favoritenstraße 9-11, A-1040 Vienna, Austria  (eduard.eiben@tuwien.ac.at)} 
\and
Michal Kotrb\v{c}\'ik\thanks{Department of Computer Science, Faculty of Informatics, Masaryk University, Botanick\'a 68a, 602 00 Brno, Czech Republic (kotrbcik@fi.muni.cz)
}
}

\maketitle

\begin{abstract}
A graph is equimatchable if each of  its matchings is a subset of a maximum matching. It is  known that any $2$-connected equimatchable graph is either bipartite, or factor-critical, and that these two classes are disjoint. This paper provides a description of $k$-connected  equimatchable factor-critical graphs with respect to their $k$-cuts for $k\ge 3$. 
As our main result we prove that if $G$ is a $k$-connected equimatchable factor-critical graph with at least $2k+3$ vertices and a  $k$-cut $S$, then $G-S$ has exactly two components and both these components are  close to being complete or complete bipartite. If both components of $G-S$ additionally have  at least $3$ vertices and $k\ge 4$, then the graph has independence number $2$. On the other hand, since every $2$-connected odd graph with independence number 2 is equimatchable,  we get the following result. For any $k\ge 4$ let $G$ be a  $k$-connected odd graph with at least $2k+3$ vertices and  a $k$-cut $S$ such that $G-S$ has two components  with  at least $3$ vertices. Then $G$ has independence number $2$ if and only if it is equimatchable and factor-critical. 
Furthermore, we show that a $2$-connected odd graph  $G$ with at least 4 vertices has independence number at most 2 if and only if $G$ is equimatchable and factor-critical and $G+e$ is equimatchable for every edge of the complement of $G$.
\end{abstract}


\section{Introduction}
A graph is \emph{equimatchable} if each of  its maximal matchings is maximum. 
Equimatchable graphs were introduced in \cite{grunbaum}, \cite{lewin},  and \cite{meng}; in particular Gr\"unbaum \cite{grunbaum}  asked  for a characterisation of all equimatchable graphs.
 If equimatchable graphs are  required to have a perfect matching, the answer turns out to be  fairly simple -- $K_{2n}$ and $K_{n,n}$ for all $n$ are the only such graphs, see \cite{sumner}. A general description of all equimatchable graphs in terms of their Gallai-Edmonds decomposition is provided in \cite{LPP}. Particular consequences of this description are that there is a polynomial-time algorithm recognizing equimatchable graphs, and that every $2$-connected equimatchable graph is either bipartite, or factor-critical. 
On the other hand, if the graph is $2$-connected, then the Gallai-Edmonds decomposition provides no additional anformation about the structure of the graph.
Since these early results, a significant attention was given to equimatchable graphs and related concepts of extendability, see \cite{plummer:1994}, \cite{plummer:2007}, and \cite{plummer:2008} for  surveys of the area. 
Despite considerable effort, the structure of equimatchable graphs is still not very well understood.
Particular exceptions are equimatchable factor-critical graphs with 
cuts of size $1$ or $2$,
which were characterized in \cite{favaron:1986}, and planar and cubic equimatchable graphs, which were characterized in \cite{KPS}. 
The aim of this paper is to describe the structure of equimatchable factor-critical graphs with respect to their minimum vertex cuts,  extending the results of Favaron \cite{favaron:1986} to graphs with higher connectivity.  
We build on a result that for any minimal matching $M$ isolating a vertex $v$ of a 2-connected equimatchable factor-critical graph $G$ the graph $G-(V(M)\cup\{v\})$ is connected, which was used in \cite{EK:2013} to bound the maximum size of equimatchable factor-critical graphs with a given genus. 

Matchings in graphs with independence number $2$ were  studied during attempts to solve a special case of Hadwiger's conjecture, see \cite{PST:2003} and \cite{CHS:2012} for details. In particular, it is known  that any odd graph with independence number $2$ is factor-critical, see for example \cite{PST:2003}. We reveal further connections between matchings and graphs with independence number $2$.

Our main results can be described as follows.
Let $G$ be a $k$-connected equimatchable factor-critical graph with a $k$-cut $S$, 
where $k\ge 3$.
If $G-S$ has at least $2k+3$ vertices,  then $G-S$ has exactly two components 
and both these components
are very close to being complete or complete bipartite. If both components of $G-S$ additionally have  at least $3$ vertices, 
then both are complete. Furthermore, if we also require 
 $k\ge 4$, then the graph has independence number $2$. 
On the other hand, we show that
every $2$-connected odd graph with independence number 2 is equimatchable and thus we get the following result. For any $k\ge 4$ let $G$ be a  $k$-connected odd graph with at least $2k+3$ vertices and  a $k$-cut $S$ such that $G-S$ has two components with  at least $3$ vertices. Then $G$ is equimatchable and factor-critical if and only if it has independence number $2$.
It turns out that independence number is related with equimatchable graphs also in the following way. 
A $2$-connected odd graph  $G$ with at least 4 vertices has independence number at most 2 if and only if $G$ is equimatchable and factor-critical and $G+e$ is equimatchable for every edge $e$ of the complement of $G$.

\section{Preliminaries}
All graphs in this paper are  finite, undirected, and  simple. 
All subgraphs are considered to be induced subgraphs unless immediately evident otherwise.
If $X$ is a set and $x$ an element of $X$, for brevity we denote  the set obtained by removing $x$ from $X$
by $X-x$. If $G$ is a graph and $v$ a vertex of $G$, with a slight abuse of notation we denote by $G-v$ the subgraph of $G$ induced by $V(G)-v$.
We say that  an edge is \emph{between} $A$ and $B$ if it has one endpoint in $A$ and the other endpoint in $B$, where 
 $A$ and $B$ are subgraphs, or sets of vertices, of a graph $G$. 
 Similarly, a set of edges or a matching are \emph{between} $A$ and $B$ if all their edges are between $A$ and $B$.
A graph or a component is \emph{even} if it has even number of vertices, otherwise it is \emph{odd}.
By a \emph{cut} we always mean a vertex cut.
A graph is \emph{randomly matchable} if it is equimatchable and has a perfect matching; it is known that a graph is connected and randomly matchable if and only if it is isomorphic with $K_{2n}$ or $K_{n,n}$ for some positive integer $n$, see \cite{sumner}.
For a matching $M$ of a graph $G$, by $V(M)$ we denote the vertices of $G$ covered by the edges of $M$.
We say that a matching $M$ \emph{isolates} a vertex $v$ of $G$ if $\{v\}$ is a component of $G-V(M)$. A matching $M$ is a \emph{minimal isolating matching of $v$} if $M$ isolates $v$ and
no proper subset of $M$ isolates $v$. We repeatedly use the following result.

\begin{theorem}[Eiben and Kotrb\v{c}\'ik \cite{EK:2013}]
\label{thm:paper1}
Let $G$ be a 2-connected  equimatchable factor-critical graph. Let $v$ be a vertex of $G$ and $M_v$ a minimal matching  isolating $v$.
Then $G-(V(M_v)\cup \{v\})$ is connected and randomly matchable.
\end{theorem}

We assume that the reader is familiar with  basic properties of matchings;  for more details we refer to  \cite{LP}.

\section{Vertex cuts in equimatchable factor-critical graphs}

The aim of this section is to describe the structure of equimatchable factor-critical graphs with respect to their minimum vertex cuts. 
Favaron \cite{favaron:1986} provided a  characterisation of equimatchable factor-critical graphs with connectivity 1 or 2 with respect to their minimum vertex cuts.

\begin{theorem}[Favaron \cite{favaron:1986}]
A graph $G$ with vertex-connectivity $1$ is equimatchable and factor-critical if and only if all of the following conditions hold:\\
\textit{(1)} $G$ has  exactly one cut-vertex $d$;\\
\textit{(2)} every connected component $C_{i}$ of $G - d$ is randomly matchable; and\\
\textit{(3)} $d$ is adjacent to at least two adjacent vertices of each $C_{i}$.
\end{theorem}

While the  case of  graphs with connectivity 1 is somewhat exceptional, our results for connectivity $k\ge 3$ are in nature very similar to Theorem~\ref{thm:favaron-2conn} below. In particular, the difficulties with describing the whole larger component in the case when the smaller component is a singleton carry completely to large connectivity, as can be seen from Theorem~\ref{thm:k-and-one}.

\begin{theorem}[Favaron  \cite{favaron:1986}]
\label{thm:favaron-2conn}
    Let $G$ be a $2$-connected equimatchable factor-critical graph with at least $4$ vertices and
    a $2$-cut $S=\{s_{1}, s_{2}\}$. Then $G- S$ has precisely two components, one of them even
    and the other odd. 
    Let $A$ and $B$ denote the even, respectively the odd component of $G- S$, 
    let $a_{1}$ and $a_{2}$ be two distinct vertices of $A$ adjacent to $s_{1}$ and  $s_{2}$, respectively, 
    and, if $|B|> 1$, let $b_{1}$ and $b_{2}$ be two distinct vertices of $B$ adjacent to $s_{1}$ and $s_{2}$, respectively.
    Then $G$ has the following structure:\\
\textit{(1)} $B$ is one of the four graphs $K_{2p+1}$, $K_{2p+1}- \{b_{1}b_{2}\}$, $K_{p,p+1}$, $K_{p,p+1}\cup \{b_{1}b_{2}\}$ for some nonegative integer $p$.
In the two last cases all neighbours of $S$ in $B$ belong to the larger partite set of $K_{p,p+1}$.\\
\textit{(2)} $A- \{ a_{1} ,a_{2}\}$ is connected randomly matchable
and, if $|B| > 1$, then $A$ is connected randomly matchable.
\end{theorem}

We extend these results to arbitrary fixed connectivity $k\ge 3$ by showing that if the graph has at least $2k+3$ vertices,
then there are exactly two components and both these components are almost complete or complete bipartite.
Our point of departure is  a lemma which allows us to efficiently apply Theorem~\ref{thm:paper1} to bound the number of components. 

\begin{lem}\label{lem:matchingCut}
    Let $G$ be a $2$-connected equimatchable factor-critical graph, $M$ a matching of $G$, and 
    $H$ an odd component of $G- V(M)$. Then $G- (H \cup V(M))$ 
    is connected randomly matchable.
\end{lem}

\begin{proof}
    Since $G$ is equimatchable, the matching $M$ can be extended to a maximum matching $M'$ of $G$. The fact that $G$ is factor-critical implies that $M'$ leaves uncovered exactly one vertex $v$ of $G$.
Clearly, $M'$ cannot cover all vertices of $H$ and hence $v$ lies in $H$.
The matching $M'$ covers all neighbours of $v$ and thus it is an isolating matching of $v$. 
Consider any minimal matching $M_v$ such that $M_{v}\subseteq M'$ and $M_v$ isolates $v$. Let $G'$ denote the graph $G- (V(M_v)\cup \{v\})$.
By Theorem \ref{thm:paper1} the graph $G'$ is connected randomly matchable.
It is not difficult to see that $M_v$ can contain only edges of $M$ and edges of $H$, and thus 
$\{v\}\cup V(M_v)\subseteq H\cup V(M)$. 
It follows that
$G- (H \cup V(M)) \subseteq G- (\{v\} \cup V(M_v)) = G'$ and that
 the graph $G- (H \cup V(M))$ can be obtained from $G'$ by removing the vertices covered by the edges of $M- M_v$. It is easy to see that removing any two adjacent vertices
of $K_{2n}$ or $K_{n,n}$ 
leads to $K_{2n-2}$ or $K_{n-1,n-1}$.
We conclude 
that $G- (H \cup V(M))$ is connected randomly matchable,
as claimed.
\end{proof}

The next lemma guarantees the existence of a large number of independent edges between 
any subset of a cut and a component separated by the cut.

\begin{lem}
\label{lem:independentEdges}
    Let $G$ be a 
	$k$-connected graph
with a $k$-cut $S$, where $k\ge 0$.
    Let $H$ be a component of $G- S$. Then for arbitrary set of vertices $X\subseteq S$ the graph $G$
contains 
at least $\min(|H|, |X|)$ independent edges between $H$ and $X$.
\end{lem}

\begin{proof}
    We prove the lemma by contradiction. 
    Let $l$ be the maximum number of independent edges of $G$ between $H$ and $X$ and suppose that $l < \min(|H|, |X|)$. 
    Since any set of independent edges between $H$ and $X$ is a matching between the vertices of $H$ and $X$, any maximum matching between $H$ and $X$ has size $l$.
    By K\"onig's theorem \cite{konig} the maximum size of a matching between $H$ and $X$ equals the minimum 
    cardinality of a vertex cover of all edges between $H$ and $X$. Hence there is a vertex set $Y\subseteq (H\cup X)$ such that
    $|Y|=l$ and $Y$ cover all edges between $H$ and $X$. 
    Since $|Y|<|H|$, the set $H- Y$ contains at least one vertex and $(S- X)\cup Y$ is a vertex cut of $G$. 
    Using $|Y|<|X|$ we get that  the size of $(S- X)\cup Y$ satisfies $(|S| - |X|) + |Y|= k-|X|+|Y|<k$, which contradicts  the fact that 
    $G$ is $k$-connected. 
\end{proof}

We are now ready to prove that
in the case where
there is a component with at least $k$ vertices
and a component with precisely one vertex
there are exactly two components and
 the larger component, except the vertices matched with the cut, is complete or complete bipartite.
 However, as stated earlier, a description of the structure of the graph induced on $V(M)$ and of the edges between $V(M)$ and $C$ seems to be quite difficult and remains to be an open problem.

\begin{theorem}
\label{thm:k-and-one}
Let $G$ be a $k$-connected equimatchable factor-critical graph with a $k$-cut $S$ such that $G- S$ has a component with  a single vertex and a component with at least $k$ vertices, where $k\ge 2$.
Then $G-S$ has exactly two components and there is a matching $M$ between $S$ and $C$ covering all vertices of $S$. Furthermore, $C-V(M)$ is connected randomly matchable.
\end{theorem}

\begin{proof}
Existence of a matching $M$ between $S$ and $C$ covering all vertices of $S$ is a consequence of Lemma~\ref{lem:independentEdges}. 
Let $v$ be the vertex of the single-vertex component of $G-S$.
Lemma~\ref{lem:independentEdges}
  implies that $v$ is adjacent to every vertex of $S$ and thus $M$ is a minimal isolating matching of $v$. 
By Theorem~\ref{thm:paper1} the graph $G-(V(M)\cup \{v\})$ is connected and randomly matchable,
which completes the proof.
\end{proof}

The next lemma implies that if the graph has at least $2k+3$ vertices, then removing any minimum cut yields precisely two components. 

\begin{lemma}
\label{lemma:two-components}
Let $G$ be a $k$-connected equimatchable  factor-critical graph with 
a $k$-cut $S$, where $k\ge 2$.
If $G$ has at least $2k+3$ vertices, then $G-S$ has precisely two components.
\end{lemma}

\begin{proof}
For a contradiction suppose that $H_1,\ldots, H_l$ are the components of $G- S$ for some $l\ge 3$.
Let $M$ be a matching between $S$ and $H_1\cup H_2$ covering as many vertices of $S$ as possible while 
leaving uncovered odd number of vertices of both $H_1$ and $H_2$.
Observe that such a matching always exists since $k\ge 2$ and, by Lemma \ref{lem:independentEdges}, every vertex of $S$ is adjacent to every component of $G-S$.
First we prove that if $M$ leaves uncovered at least $2$ vertices of $S$, then it leaves uncovered precisely one vertex in both $H_1$ and $H_2$. Indeed, suppose for the contrary that $M$ leaves uncovered at  
least two vertices $s_1$ and $s_2$ of $S$ and more than one vertex in, say, $H_1$. Note that in this case $M$ leaves uncovered at least 3 vertices of $H_1$.
Denote by  $M_1$ the edges of $M$ incident with $H_1$. Let $X = \{s_1,s_2\}\cup(S\cap V(M_1))$. Applying Lemma \ref{lem:independentEdges} to $H_1$ and $X$ yields that there is a matching $M'$ between $H_1$ and $X$ covering all vertices of $X$. 
It can be easily seen that $M'' = M'\cup (M- M_1)$  is a matching between $S$ and $H_1\cup H_2$ which leaves uncovered odd number of vertices in both $H_1$ and $H_2$, and that $M''$ is larger than $M$, which contradicts the maximality of $M$.

We proceed to extend $M$ to a matching $N$ between $S$ and $G-S$ such that $N$ covers all vertices of $S$ and leaving uncovered odd number of vertices in both $H_1$ and $H_2$. If $M$ covers all vertices of $S$, then let $N=M$.
If $M$ leaves uncovered precisely one vertex $s$ of $S$, then let $N=M \cup\{e\}$, where $e$ is any edge joining $s$ with $H_3$, note that such an edge always exists by Lemma~ \ref{lem:independentEdges}. Finally, if $M$ leaves uncovered at least $2$ vertices of $S$, then it leaves uncovered exactly one vertex in both $H_1$ and $H_2$ as shown above, and $|V(G)|\ge 2k+3$ implies that $H_3\cup\dots\cup H_l$ contains more vertices than $S-V(M)$. Therefore, by Lemma~\ref{lem:independentEdges} there is a matching $N'$ between $S-V(M)$ and $H_3\cup \dots\cup H_l$ covering all vertices of $S-V(M)$. Now $N=M\cup N'$ is the desired matching covering all vertices of $S$ and leaving uncovered odd number of vertices in both $H_1$ and $H_2$.

To complete the proof it suffices to show that $N$ cannot be extended to a maximum matching of $G$, contradicting the fact that $G$ is equimatchable. Indeed, $N$ leaves uncovered odd number of vertices in both $H_1$ and $H_2$ and separates $H_1$ and $H_2$ from the rest of the graph and thus any maximal matching $N''\supseteq N$ leaves uncovered at least one vertex in both $H_1$ and $H_2$. Since $G$ is equimatchable and factor-critical, any maximum matching of $G$ leaves uncovered precisely one vertex of $G$ and hence $N''$ cannot be a maximum matching. The proof is now complete.
\end{proof}

To deal with the cases where the smaller component of $G-S$ has at least two vertices we will 
need the following lemma.

\begin{lemma}
\label{lemma:k-and-two}
Let $G$ be a $k$-connected equimatchable factor-critical graph with a $k$-cut $S$, where $k\ge 2$.
Assume that $G- S$ has a component $C$ with at least $k$ vertices and $G- (S\cup C)$ has a component with exactly two vertices.
Then $G-S$ has exactly two components and there is a matching $M$ between $S$ and $C$ covering all vertices of $S$. Furthermore, 
for any matching $M'$ between $S$ and $C$ covering all vertices of $S$ and 
for each vertex $x$ of $C\cap V(M')$, the subgraph of $G$ induced by $(C- V(M'))\cup \{x\}$ is connected randomly matchable.
\end{lemma}

\begin{proof}
Existence of a matching $M$ between $S$ and $C$ covering all vertices of $S$ is a consequence of Lemma~\ref{lem:independentEdges}. 
Let $M$ be any matching between $S$ and $C$ covering all vertices of $S$ and let $D$ be a component of $G- (C\cup S)$ with exactly two vertices.
Let $x$ be any vertex of $C$ covered by $M$ and let $s$ be the vertex of $S$ matched by $M$ with $x$.
Lemma~\ref{lem:independentEdges}
  implies that there is a vertex of $D$, say $d$, adjacent to $s$. Let $d'$ be the vertex of $D$ different from $d$. 
  Consider the set $M' = (M- \{sx\})\cup \{ds\}$; clearly $M'$ is a matching and  $\{d'\}$ is an odd component of $G- V(M')$.
 Thus 
by Lemma~\ref{lem:matchingCut} the graph $G-(V(M')\cup \{d'\})=(C- V(M))\cup \{x\}$ is connected randomly matchable,
which completes the proof.
\end{proof}

The following theorem provides a characterisation of $k$-connected equimatchable factor-critical graphs with a $k$-cut $S$ such that $G-S$ contains a component with at least $k$ vertices and a component with precisely $2$ vertices. We indicate the end of a proof of a claim by $\blacksquare$.

\begin{theorem}
\label{thm:kCutEFC}
Let $G$ be a $k$-connected equimatchable factor-critical graph with a $k$-cut $S$, where $k\ge 3$.
Assume that $G- S$ has a component $C$ with at least $k$ vertices and $G- (S\cup C)$ has a component with exactly two vertices.
Then $G- S$ has exactly two components. Furthermore,  if $S$ contains an edge, then $C$ is a complete graph.
If $S$ does not contain an edge,  
then there is a nonegative integer $m$ and sets $\{x_1,\ldots, x_m\}$ of vertices of $C$ and  $\{y_1,\ldots, y_m\}$ of vertices of $S$ such that  $x_iy_i$ is not an edge of $G$ for every $i\in \{1,\ldots, m\}$ and $C\cup S\cup \{x_1y_1,\ldots, x_my_m\}$ is isomorphic with $K_{n,n+1}$ for some $n$.
\end{theorem}

\begin{proof}
Let $S=\{s_1,\ldots, s_k\}$ and
let $D=\{d_1,d_2\}$ be a component of $G- (C\cup S)$ with exactly two vertices.
Note that since $G$ is odd, it has at least $2k+3$ vertices and hence, by  Lemma~\ref{lemma:two-components},
$D$ is the only component of $G- (C\cup S)$. 
By Lemma~\ref{lem:independentEdges} there is a matching $M$ between $S$ and $C$ which covers all vertices of $S$. Denote by $X$ the set $C\cap V(M)$ and let $C'=C-X$. 
The fact that $G- S$ has exactly two components follows from Lemma~\ref{lemma:k-and-two}.
The rest of the proof is split into two cases.

\vspace*{2mm}
\noindent {\bf Case A) } There is an edge in $S$.

\begin{claimA}
\label{claimA:triangle-in-X}
 If $rs$ is an edge in $S$ and $u$ and $v$ are the two vertices of $C$ matched by $M$ with $r$ and $s$, respectively, then $\{u,v,w\}$ is a triangle for any vertex $w$ of $X- \{u,v\}$.
\end{claimA}

\begin{cProofA}
	Choose an arbitrary vertex $w$ from $X- \{u,v\}$ and let $t$ be the vertex of $S$ matched by $M$ to $w$.
By Lemma~\ref{lem:independentEdges} there is an edge between $s$ and a vertex of $D$, say $d$. Let $M' = (M-\{ru,sv,tw \})\cup\{rs,td\}$ and denote the only vertex of $D-d$ by $d'$.
Clearly, $M'$ is a matching of $G$ which isolates $d'$.  
Let $x$ be any vertex from $\{u,v,w\}$. 
Applying Lemma~\ref{lemma:k-and-two} to
$C'\cup\{x\}$ and $M$ yields that $C'\cup\{x\}$ is
randomly matchable
and thus 
it has
a perfect matching $M_x$. 
Observe that the set $M_x\cup M'$ is a matching of $G$ which leaves uncovered precisely the vertices in $\{d', u, v, w\}-\{x\}$. Because $\{u,v,w\}\subseteq C$ and 
$S$ is a cut separating $C$ and $D$, there
is no edge between $d'$ and $\{u,v,w\}$.
The  fact that $G$ is equimatchable and factor-critical implies that the two vertices in $\{u, v, w\}-\{x\}$ are joined by an edge. Since $x$ was arbitrary vertex from $\{u,v,w\}$, the claim follows.
\hfill $\blacksquare$
\end{cProofA}

 If $k=3$, then the result follows from Claim~A\ref{claimA:triangle-in-X}. Therefore, from now on we assume $k\ge 4$.

\begin{claimA}
\label{claimA:X-is-complete}
The subgraph of $G$ induced by $X$ is a complete graph.
\end{claimA}
\begin{cProofA}
Let $rs$ be an edge of $S$.
Our aim is to show that there is an edge between  arbitrary two vertices $y$ and $z$ of $X$.
Denote by $x_r$ and $x_s$ the two vertices of $X$ joined by $M$ to $r$ and $s$, respectively. If $y$ or $z$ belongs to $\{x_r,x_s\}$,
 then $y$ and $z$ are joined by an edge by Claim~A\ref{claimA:triangle-in-X}. Hence we can assume that  $\{y,z\}\cap \{x_r,x_s\} = \emptyset$. Claim~A\ref{claimA:triangle-in-X} applied to $\{x_r,x_s,y\}$ shows that $x_ry$ is an edge of $G$. 
Applying Lemma~\ref{lemma:k-and-two} 
to $C'\cup\{z\}$ and $M$ yields that $C'\cup\{z\}$ is  randomly matchable and thus it has a perfect matching $M'$. 
Let $s_y$ and $s_z$ be the vertices of $S$ joined by $M$ to $y$ and $z$, respectively.
Consider the set
$M'' = (M-\{rx_r,ys_y,zs_z\})\cup M' \cup \{yx_r, e\}$, where $e$ is the edge in $D$. It is not difficult to see that $M''$ is a matching which leaves uncovered exactly the vertices $r$, $s_y$, and $s_z$. Hence $\{r,s_y,s_z\}$ contains an edge $e$ and the result follows by using Claim~A\ref{claimA:triangle-in-X} on $e$ and  $\{x_r,y,z\}$.
\hfill $\blacksquare$
\end{cProofA}

\begin{claimA}
\label{claimA:C'-is-complete}
The subgraph of $G$ induced by $C'$ is a complete graph.
\end{claimA}

\begin{cProofA}
 Assume that the edge in $S$ is $rs$. 
Let $x_r$ and $x_s$ be the vertices of $X$ joined by $M$ to $r$ and $s$, respectively, and let $y$ be an arbitrary vertex of $X-\{x_r,x_s\}$. By Claim~A\ref{claimA:triangle-in-X} applied to $rs$ the subgraph of $G$ induced by $\{y,x_r,x_s\}$ is a triangle. Let $s_y$ be the vertex of $S$ joined to $y$ by $M$. By Lemma~\ref{lem:independentEdges} there is a vertex of $D$, say $d$, adjacent to $s_y$. Let $d'$ be the vertex of $D$ different from $d$. Consider the set $M' = M-\{rx_r, sx_s, ys_y\}\cup\{rs,ds_y\}$; clearly $M'$ is a matching isolating $d'$. By Theorem~\ref{thm:paper1} the graph $G-(V(M)\cup\{d'\}) =C'\cup\{y,x_r,x_s\}$ is either $K_{2n}$ or $K_{n,n}$. 
Since  $\{y,x_r,x_s\}$ induces a triangle and is contained in $C'\cup\{y,x_r,x_s\}$,  the graph $C'\cup\{y,x_r,x_s\}$ is a complete graph. In particular $C'$ is a complete graph, as claimed.
\hfill $\blacksquare$
\end{cProofA}

\begin{claimA}
 The subgraph of $G$ induced by $C$ is a complete graph.
\end{claimA}

\begin{cProofA}
 By Claim~A\ref{claimA:X-is-complete} the set $X$ induces a complete graph and by Claim~A\ref{claimA:C'-is-complete} the set $C'$ induces a complete graph. Lemma~\ref{lemma:k-and-two} 
 implies that for each vertex $x$ of $X$ the graph $C'\cup \{x\}$ is connected randomly matchable. It is not difficult to see that if $C'$ is a complete graph, then also $C'\cup \{x\}$ is a complete graph for each $x$ of $X$. It follows that each vertex of $X$ is adjacent to every vertex of $C'$ and thus $C$ is a complete graph, as claimed.
\hfill $\blacksquare$
\end{cProofA}

The preceding claim completes the case where there is an edge in $S$ and   the first part of the proof.

\vspace*{2mm}
\noindent {\bf Case B)} The set $S$ is independent.
\vspace*{2mm}

\begin{claimB}
\label{claimB:X-independent}
The set $X$ is independent.
\end{claimB}

\begin{cProofB}
For a contradiction suppose that 
  $x_1x_2$ is an edge in $X$ and let $x_3$ be an arbitrary vertex of $X-\{x_1,x_2\}$. Let $M'$ be an arbitrary perfect matching of $C'\cup\{x_3\}$, due to Lemma~\ref{lemma:k-and-two} such a matching exists. 
Furthermore, let $s_i$ be the vertex matched by $M$ with $x_i$ for $i=1,2,3$ and let $e$ be the edge in $D$.
The  matching 
$M'' = (M- \{s_1x_1, s_2x_2, s_3x_3\})\cup M' \cup \{x_1x_2,e\}$ 
leaves uncovered only the vertices $s_1$, $s_2$, and $s_3$. By the assumption of Case B) the maching $M''$ is maximal, contradicting the fact that $G$ is equimatchable and factor-critical.
\hfill $\blacksquare$
\end{cProofB}

\begin{claimB}
\label{claimB:Knn+1}
 The subgraph of $G$ induced by $C'$ is isomorphic with $K_{n+1,n}$ for some $n\ge 0$.
\end{claimB}

\begin{cProofB}
If $C'$ contains only one vertex, then the claim holds. Since $C'$ is odd, we can assume  $|V(C')|\ge 3$.
 Lemma~\ref{lemma:k-and-two}  implies that  for each 
$x\in X$
the graph $C'\cup\{ x\}$ is connected randomly matchable.
If $C'\cup \{x\}$ is $K_{m,m}$ for some $x\in X$, then $C'$ is clearly $K_{m,m-1}$ and the claim holds.
For a contradiction suppose that $C'\cup \{x\}$ is $K_{2m}$ for each $x\in X$.
Let $x$ be an arbitrary vertex of $X$, let $M'$ be a perfect matching of $C'\cup \{x\}$, and $bc$ an edge of $M'$ not incident with $x$. 
Observe that both $b$ and $c$ are adjacent with 
each vertex $x'$ of $X$
 since $C'\cup \{x'\}$ is $K_{2m}$.
Let $x_b$ and $x_c$ be two vertices of $X-x$. 
Finally, let $s_x$, $s_b$, and $s_c$, be the vertex matched by $M$ with $x$, $x_b$, respectively $x_c$. 
It follows that
the set $M''$ defined by $M''=(M- \{s_xx,s_bx_b, s_cx_c \})\cup (M'-\{bc \})\cup \{ x_bb,x_cc, d_1d_2\}$ is a matching which covers all vertices of $G$ except $s_x$, $s_b$, and $s_c$.
Because $S$ is independent, $M''$ is a maximal matching leaving uncovered $3$ vertices, which contradicts the fact that $G$ is equimatchable factor-critical and completes the proof of the claim.
\hfill $\blacksquare$
\end{cProofB}

\noindent Denote by $U$ the smaller and by $W$ the larger partite set of $C'$.

\begin{claimB} 
\label{claimB:no-edge-X-U}
There is no edge between $X$ and $U$.
\end{claimB}

\begin{cProofB}
We proceed by contradiction. Suppose that $u$ is a vertex of $U$ adjacent to a vertex $x$ of $X$. Let $s$ be the vertex of $S$ matched with $x$ by $M$ and let $d$ be any vertex of $D$ adjacent to $s$; such a vertex $d$ exists by Lemma~\ref{lem:independentEdges}. Clearly, the set $M'= (M-\{sx\}) \cup \{ds, xu\}$ is a matching of $G$. It is not difficult to see that any maximal matching containing $M'$ leaves unmatched at least two vertices of $W$, which contradicts the fact that $G$ is equimatchable and factor-critical.
\hfill $\blacksquare$
\end{cProofB}

\begin{claimB}
\label{claimB:no-edge-S-W}
There is no edge between $S$ and $W$.
\end{claimB}
\begin{cProofB}
 For a contradiction suppose that there is a vertex $w$ of $W$ adjacent to some vertex $s$ of $S$.
Let $t$ be any vertex of $S-s$ and let $d$ be a vertex of $D$ adjacent to $t$, such a vertex $d$ exists by 
Lemma~\ref{lem:independentEdges}. Furthermore, let $x_s$ and $x_t$ be the vertices of $X$ matched by $M$ with $s$ and $t$ respectively, and 
let $N=(M- \{sx_s,tx_t\})\cup \{sw, td\}$. Clearly, $N$ is a matching of $G$. 
Note that $N$ leaves uncovered $D-d$ and $(C'\cup \{x_s,x_t\})-w$. Since $S$ is a cut, the vertex in $D-d$ is not adjacent with any vertex in $(C'\cup \{x_s,x_t\})-w$. 
Claim~B\ref{claimB:Knn+1} and the choice of $w$ imply that $C'-w$ is $K_{n,n}$ for some $n$. Furthermore, by Claim~B\ref{claimB:no-edge-X-U} there is no edge between $\{x_s,x_t\}$ and $U$. It follows that 
$(C'\cup \{x_s,x_t\})-w$ is a subgraph of $K_{n+2,n}$ and thus any maximal matching of $G$ containing $N$ leaves uncovered $d$ and at least two vertices of $(C'\cup \{x_s,x_t\})-w$, contradicting the fact that $G$ is equimatchable factor-critical.
\hfill $\blacksquare$
\end{cProofB}

\begin{claimB}
\label{claimB:all-edges-X-W}
 Each vertex of $X$ is adjacent to every vertex of $W$.
\end{claimB}

\begin{cProofB}
 Let $x$ be a vertex of $X$ and $w$ a vertex of $W$. By Claim~B\ref{claimB:Knn+1} the graph $C'$ is $K_{n,n+1}$ and by 
the definition of $W$ the vertex 
$w$ lies in the larger partite set of $C'$. It follows that there is a perfect matching $M'$ of $C'-\{w\}$. Let $s$ be the vertex of $S$ matched with $x$ by $M$. By Lemma~\ref{lem:independentEdges} there is an edge $e$ between $s$ and $D$. Let $d$ be the vertex of $D$ not covered by $e$. Let $M'' = (M-\{xs\})\cup M' \cup \{e\}$. Clearly, $M''$ is a matching which covers all vertices of $G$ except $d,x$ and $w$. Since $C$ and $D$ are different components of $G-S$ and $x$ and $w$ lie in $C$, the vertex $d$ is adjacent with neither $x$, nor $w$. Using the fact that $G$ is factor-critical and equimatchable we get that $x$ and $w$ are adjacent, which completes the proof.
\hfill $\blacksquare$
\end{cProofB}

\begin{claimB}
\label{claimB:almost-all-edges-S-X-U}
Each vertex of $S$ is adjacent to either all, or all but one vertices of $X\cup U$.
\end{claimB}
\begin{cProofB}
Suppose to the contrary that there is a vertex $s$ of $S$ and two vertices $v_1$ and $v_2$ from $X\cup U$ such that $s$ is adjacent neither to $v_1$, nor to $v_2$. Let $x$ be the vertex of $X$ matched with $s$ by $M$ and note that $x$ is different from both $v_1$ and $v_2$.
If $v_1\in X$, then let $y_1 = v_1$, otherwise let $y_1$ be an arbitrary vertex of $X-\{x, v_2\}$. Similarly, if $v_2\in X$, then let $y_2 = v_2$, otherwise let $y_2$ be an arbitrary vertex of $X-\{x,y_1\}$. 
Let $t_1$ and $t_2$ be the two vertices of $S$ matched by $M$ with $y_1$, respectively $y_2$.
Let $M'$ be a set of two independent edges between $D$ and $\{t_1, t_2\}$; such two edges exist by  Lemma~\ref{lem:independentEdges}. Recall that 
 the graph $C'$ is isomorphic with $K_{n,n+1}$ by Claim~B\ref{claimB:Knn+1}
 and that each vertex of $X$ is adjacent with every vertex of $W$ 
by Claim~B\ref{claimB:all-edges-X-W}.
Using the last two observations it is not difficult to prove that  $C'\cup\{x,y_1,y_2\}$ has a matching $N$ which covers all vertices of $C'\cup\{x,y_1,y_2\}$ except $v_1$ and $ v_2$; a straightforward case analysis on $|\{y_1,y_2\}\cap \{v_1,v_2\}|$ is left to the reader.
Consider the set $N' = (M- \{sx, y_1t_1, y_2t_2\})\cup M'\cup N$. 
It is easy to see that $N'$ is a matching which covers all vertices of $G$ except $v_1,v_2$, and $s$.
Observe that there is no edge between $v_1$ and $v_2$. Indeed, if one of $v_1,v_2$ belongs to $X$ and the other to $U$, then they are not adjacent by 
Claim~B\ref{claimB:no-edge-X-U}. If both $v_1$ and $v_2$ are from $X$, then they are not adjacent by Claim~B\ref{claimB:X-independent}. Finally, if both $v_1$ and $v_2$ are from $U$, then they are not adjacent by the definition of $U$.
Since by our assumption $s$ is adjacent with neither $v_1$, nor $v_2$, we get a contradiction with the fact that $G$ is equimatchable and factor-critical.
\hfill $\blacksquare$
\end{cProofB}

\begin{claimB}
\label{claimB:almost-all-edges-X-S}
Each vertex  of $X$ is adjacent to either all, or all but one vertices of $S$.
\end{claimB}

\begin{cProofB}
Suppose for the contrary that there is a vertex $x$ of $X$ and two vertices $t_1$ and $t_2$ of $S$ such that $x$ is adjacent to neither $t_1$, nor $t_2$.
 Let $s$ be the vertex of $S$ matched with $x$ by $M$ and let $y_1$ and $y_2$ be the two vertices matched by $M$ with $t_1$ and $t_2$, respectively. By Claim~B\ref{claimB:almost-all-edges-S-X-U} the vertex $s$ is adjacent with at least one of  $y_1$ and $y_2$; without loss of generality we  assume that $s$ is adjacent to $y_1$.
By Lemma~\ref{lemma:k-and-two} the graph $C'\cup\{y_2\}$ is randomly matchable and hence it has a perfect matching $M'$.
Let $M'' = (M- \{sx,t_1y_1,t_2y_2\})\cup M' \cup \{e, sy_1\}$, where $e$ is the edge in $D$.
It is not difficult to see that $M''$ is a matching which covers all vertices of $G$ except $x, t_1$, and $t_2$. 
By the assumption of Case B)~the vertices $t_1$ and $t_2$ are not adjacent and by our assumption $x$ is  adjacent to neither $t_1$, nor $t_2$.
Therefore, $M''$ is a maximal matching leaving uncovered 3 vertices, contradicting
the fact that $G$ is equimatchable and factor-critical.
\hfill $\blacksquare$
\end{cProofB}

\begin{claimB}
\label{claimB:almost-all-edges-U-S}
Each vertex  of $U$ is adjacent to  either all, or all but one vertices of $S$.
\end{claimB}
\begin{cProofB}
Suppose for the contrary that there is a vertex $u$ of $U$ and two vertices $t_1$ and $t_2$ of $S$ such that $u$ is adjacent to neither $t_1$, nor $t_2$.
Let $y_1$ and $y_2$ be the two vertices of $X$ matched by $M$ with $t_1$ and $t_2$, respectively. 
By Lemma~\ref{lemma:k-and-two} the graph $C'$ is isomorphic with $K_{n,n+1}$ and by Claim~B\ref{claimB:all-edges-X-W} both vertices $y_1$ and $y_2$ are adjacent to every vertex from the larger partite set of $C'$. 
Therefore, there exists a perfect matching $M'$ of $C'\cup\{ y_1, y_2\} - \{u\}$.
Let $M'' = (M- \{t_1y_1,t_2y_2\})\cup M' \cup \{e\}$, where $e$ is the edge in $D$.
It is not difficult to see that $M''$ is a matching which covers all vertices of $G$ except $u, t_1$, and $t_2$. 
By the assumption of Case~B) the vertices $t_1$ and $t_2$ are not adjacent, and by our assumption $u$ is adjacent to neither $t_1$, nor $t_2$.
It follows that $M''$ is a maximal matching leaving uncovered 3 vertices, which contradicts the fact that $G$ is equimatchable and factor-critical. The proof of Claim~B\ref{claimB:almost-all-edges-U-S} is now complete.
\hfill $\blacksquare$
\end{cProofB}

Denote by $H$ the subgraph of $G$ induced by $C\cup S$. Claims~B\ref{claimB:no-edge-X-U} and B\ref{claimB:no-edge-S-W} imply that $U\cup W\cup X\cup S = H$ is a bipartite graph with partite sets $X\cup U$ and $S\cup W$. Claim~B\ref{claimB:Knn+1} and the definition of $U$ and $W$ yield that each vertex of $U$ is adjacent to every vertex of $W$.
By Claim~B\ref{claimB:all-edges-X-W} each vertex of $X$ is adjacent to every vertex of $W$. From Claims~B\ref{claimB:almost-all-edges-S-X-U}, B\ref{claimB:almost-all-edges-X-S}, and B\ref{claimB:almost-all-edges-U-S} we get that there is a nonnegative  integer $m$ and sets of vertices $\{t_1,\ldots, t_m\}\subseteq S$ and $\{y_1,\ldots, y_m\}\subseteq X\cup U$ such that $t_iy_i\notin E(G)$ for all $i\in \{1,\ldots, m\}$ and that $H\cup \{t_1y_1,\ldots, t_my_m\}$ is a complete bipartite graph. The proof is now complete.
\end{proof}

The following observation may be easily verified.

\begin{observation}
\label{observation:Knn-two-vertices-removal}
Let $G$ be isomorphic with $K_{n,n}$ for some $n\ge 1$ and let $u$ and $v$ be two vertices of $G$. If $G-\{u,v\}$ is randomly matchable, 
then $u$ and $v$ are adjacent.
\hfill $\square$
\end{observation}

\setcounter{claimA}{0}
\setcounter{claimB}{0}

\setcounter{cProofA}{0}
\setcounter{cProofB}{0}

\begin{theorem}
	\label{thm:Cgek-Dge3}
Let $G$ be a $k$-connected equimatchable factor-critical graph with at least $2k+3$ vertices
and a $k$-cut $S$ such that $G- S$ has two components with at least $3$ vertices, where $k\ge 3$.
    Then $G-S$ has exactly two components and both are complete graphs.
\end{theorem}

\begin{proof}
By Lemma \ref{lemma:two-components} the graph $G-S$ has precisely two components, denote these components by $C$ and $D$, respectively. 
First we deal with the case where both $C$ and $D$ are strictly smaller than $k$; this case is much simpler. 
Take any two vertices $c$ and $c'$ of a component of $G-S$, say of $C$. Let $l = |V(C)|$. Since $|V(C)| < k$, there are $l$ independent edges between $S$ and $C$ by Lemma~\ref{lem:independentEdges}. Therefore, we can choose a set $M_C$ of $l-2$ independent edges between $S$ and $C-\{c,c'\}$. Since $|V(G)|\ge 2k+3$, by Lemma~\ref{lem:independentEdges} there is a set $M_D$ of $k-l+2$ independent edges between $D$ and $S-V(M_C)$. 
Let $M=M_C\cup M_D$ and observe that $M$ is a matching of $G$.
It is not difficult to see that the vertex $c$ can be  in $G-V(M)$ adjacent only to $c'$, and similarly $c'$ can be adjacent only to $c'$. Since $G$ is equimatchable and factor-critical, the matching $M$ can be extended to a maximum matching of $G$, 
which leaves unmatched precisely one vertex of $G$. Clearly, this is possible only if $c$ and $c'$ are adjacent. Since the choice of $c$ and $c'$ was arbitrary, it follows that both components of $G-S$ are complete, as claimed.

From now on we assume that at least one component of $G-S$, say $C$, has at least $k$ vertices.
    By Lemma~\ref{lem:independentEdges} there is a set $M$ of  $k$ independent edges 
    between $C$ and $S$  covering all vertices of $S$.
Denote by $X$ the set of vertices $C\cap V(M)$ and let $C' = C - X$.
We distinguish two cases.

\medskip   
\noindent {\bf Case A) $D$ is even.} 
First observe that in this case $C'$ is odd and 
 denote by  $H$  an odd component of $C'$.
Clearly, $G$, $M$, and $H$ satisfy the assumptions of Lemma~\ref{lem:matchingCut}  which implies that $G-(H\cup V(M))$ is
connected randomly matchable.
Since $D$ is a component of $G-(H\cup V(M))$, it follows that $H$ is the only component of $C'$
and thus $H\cup V(M)=S\cup C$. Consequently, $D=G-(H\cup V(M))$ and hence $D$     is 
connected 
randomly matchable.
To prove that $D$ is complete we proceed by contradiction and suppose that $D$ is $K_{n,n}$ for some $n\ge 2$. Since $k\ge 3$, by Lemma~\ref{lem:independentEdges} there are at least three independent edges between $D$ and $S$ and at least two of these edges, say $sd$ and $s'd'$, have their endvertices in the same partite set of $D$,
where $d$ and $d'$ are vertices of $D$. 
Let $x$ and $x'$ be the vertices of $X$ matched by $M$ with $s$ and $s'$, respectively.
Let $M'=(M-\{sx,s'x'\})\cup\{sd,s'd'\}$ and let $H'$ be an odd component of $C'\cup\{x,x'\}$.
 By Lemma~\ref{lem:matchingCut} the graph $G-(H'\cup V(M'))$ is 
randomly matchable.
On the other hand, $G-(H'\cup V(M')) = D-\{d,d'\}$ and thus $d$ and $d'$ are 
adjacent by Observation~\ref{observation:Knn-two-vertices-removal}, contradicting the fact that $d$ and $d'$ lie in the same partite set.
 Therefore, we conclude that $D$ is isomorphic with $K_{2n}$.

\begin{claimA}
\label{claimA:Ccup1}
The graph $C'\cup \{x\}$ is connected randomly matchable for each $x\in X$.
\end{claimA}

\begin{cProofA}
Let $s$ be the vertex of $S$ matched by $M$ with $x$.
Since $S$ is a minimum cut, there is a vertex $d$ of $D$ adjacent with $s$. Let $H$ be an odd component of $D-d$ and let $M' =  (M\cup\{ds\})- sx$. Lemma~\ref{lem:matchingCut} applied to $H$ and $M'$ implies that $G-(H\cup V(M'))$ is connected randomly matchable and that $H= D-d$. Therefore, $C'\cup \{x\} = G-(H\cup V(M'))$ and thus $C'\cup \{x\}$ is connected randomly matchable, as claimed.
 \hfill $\blacksquare$
\end{cProofA}

\begin{claimA}
\label{claimA:Ccup3}
For each triple of pairwise distinct vertices $x,y,z$ of $X$
the graph $C'\cup \{x,y,z\}$ is isomorphic with $K_{2n}$ for some $n$.
\end{claimA}

\begin{cProofA}
Let $s_x,s_y$ and $s_z$ be the vertices matched by $M$ with $x,y,$ and $z$, respectively.
By Lemma~\ref{lem:independentEdges} there are three 
pairwise distinct vertices $d_x, d_y,$ and $d_z$ of $D$ adjacent to $s_x, s_y$, and $s_z$, respectively.
 Since $D$ is even and $|D|\ge 4$, the graph $D-\{d_x,d_y,d_z\}$ is odd and thus contains an odd component $H$.  Using  Lemma~\ref{lem:matchingCut} on $H$ and 
$
(M-\{ xs_x,ys_y,zs_z\})
\cup
\{s_xd_x,s_yd_y,s_zd_z\}
$
we get that $C'\cup \{x,y,z\}$ is connected randomly matchable. 
Since  $C'\cup \{v\}$ is connected randomly matchable for each $v\in \{x,y,z\}$ by Claim~A\ref{claimA:Ccup1}, Observation~\ref{observation:Knn-two-vertices-removal} used on all pairs from $\{x,y,z\}$ implies that $\{x,y,z\}$ induces a triangle. The last observation 
implies that 
$C'\cup \{x,y,z\}$ is 
complete
and concludes the proof of the claim.
\hfill $\blacksquare$
\end{cProofA}

Using Claim~A\ref{claimA:Ccup3} on all triples of  vertices of $X$ implies that the graph induced by $C$ is complete, as claimed.

\medskip
 \noindent   {\bf Case B) $D$ is odd.} 
    
Let $l = \min\{|D|,k \}$ and
note that
$l\ge 3$.
Our first aim is to show that $C$ is complete.

\begin{claimB}
\label{claimB:C-RM}
The subgraph of $G$ induced by $C'$ is connected randomly matchable.
\end{claimB}

\begin{cProofB}
 Lemma~\ref{lem:matchingCut} applied to $D$ and $M$ implies that $G-(D\cup V(M))$ is connected randomly matchable.
The claim follows from the fact that $C' = G-(D\cup V(M))$.
\hfill $\blacksquare$
\end{cProofB}

\begin{claimB}
\label{claimB:Ccup2-RM}
For each two vertices $x$ and $x'$ of $X$ the graph $C'\cup \{x,x'\}$ is connected randomly matchable. Furthermore, the vertices $x$ and $x'$ are adjacent.
\end{claimB}

\begin{cProofB}
 Let $s$ and $s'$ be the vertices of $S$ matched by $M$ with $x$ and $x'$, respectively.
 By Lemma~\ref{lem:independentEdges} there are two independent edges $sd$ and $s'd'$, where $d$ and $d'$ are  vertices of $D$. Since $D$ is odd and $|D|\ge 3$, there is an odd component $H$ of $D-\{d,d'\}$. 
Let $M' = (M \cup \{sd,s'd'\}) - \{ xs, x's'\}$. Clearly, $M'$ is a matching and thus Lemma~\ref{lem:matchingCut} applied to $H$ and $M'$ implies that $C'\cup \{x,x'\}$ is connected randomly matchable. If $C'\cup \{x,x'\}$  is a complete graph, then $c$ and $c'$ are adjacent and there is nothing left to prove. If $C'\cup \{x,x'\}$  is $K_{n,n}$ for some $n\ge 2$, then we get that $x$ and $x'$ are adjacent by  Claim~B\ref{claimB:C-RM} and Observation~\ref{observation:Knn-two-vertices-removal}.
\hfill $\blacksquare$
\end{cProofB}

Using Claim~B\ref{claimB:Ccup2-RM} on all pairs of vertices $x$ and $x'$ of $X$ implies that the subgraph of $G$ induced by $X$ is complete. 
Since $C'$ is even, we can thus assume that $|C'|\ge 2$.

\begin{claimB}
\label{claimB:Ccup2-complete}
 For each two vertices $x$ and $y$ of $X$ the graph $C'\cup \{x,y\}$ is $K_{2n}$ for some $n$.
\end{claimB}

\begin{cProofB}
Let $z$ be a vertex of $X-\{x,y\}$.
 By Claim~B\ref{claimB:Ccup2-RM} the graph $C'\cup \{x,y\}$ is connected randomly matchable. Suppose for a contradiction that $C'\cup \{x,y\}$ is $K_{n,n}$ for some $n\ge 2$.  
 Since $C'\cup \{x,y\}$ is $K_{n,n}$ and $n\ge 2$, Claim~B\ref{claimB:Ccup2-RM} implies that  $C'$ is $K_{n-1,n-1}$. Let $A$ and $B$ denote the partite sets of $C'$. Since $x$ and $y$ are adajcent, without loss of generality we may assume that 
the partite sets of $C'\cup \{x,y\}$ are $A\cup \{x\}$ and $B\cup\{y\}$.
In particular, $x$ is not adjacent to any vertex of $A$ and thus $C'\cup\{x,z \}$ is not a complete graph. Similarly, $C'\cup \{y,z\}$ is also not a complete graph.
Therefore, Claim~B\ref{claimB:Ccup2-RM} used on $x$ and $z$ imply that  $z$ is adjacent with all vertices of $A$ and it is not adjacent to any vertex in $B$.
However, Claim~B\ref{claimB:Ccup2-RM} used on $y$ and $z$ implies that at least one of $y$ and $z$ is adjacent to all vertices of $B$,
which is a contradiction.
\hfill $\blacksquare$
\end{cProofB}

We conclude that $C$ is complete by using Claim~B\ref{claimB:Ccup2-complete} on all pairs of vertices $x$ and $x'$ of $X$.

\medskip
Now we prove that $D$ is complete. By Lemma~\ref{lem:independentEdges} 
there is a set of $l$ independent edges 
$\{s_1d_1,\ldots, s_ld_l\}$ between $S$ and $D$, where $d_1,\ldots, d_l$ are vertices of $D$. For each $i\in \{1,\ldots, l\}$ denote the graph $D-d_i$ by $D_i$ and let $x_i$ be the vertex of $X$ matched by $M$ with $s_i$.

\begin{claimB}
\label{claimB:Di-RM}
For each $i\in \{1,\ldots, l\}$ the graph $D_i$ is connected randomly matchable.
\end{claimB}

\begin{cProofB}
Let $M_{i}=(M-\{s_{i}x_i\})\cup \{d_{i}s_i\}$ and
 let $H_{i}$ be an odd component of  
$C'\cup\{x_i\}$. 
Lemma~\ref{lem:matchingCut} applied to $G$, $M_{i}$, and $H_{i}$ yields that $G-(H_{i}\cup V(M_{i}))$ is connected randomly matchable  and thus  $H_{i}$ is the only component of $C'\cup\{x_i\}$.
Consequently, $D_i = G-(H_{i}\cup V(M_{i}))$ and thus $D_i$ is connected randomly matchable, as claimed.
\hfill $\blacksquare$
\end{cProofB}

Since $l\ge 3$, it is easy to see that if $D$ contains only three vertices, then Claim~B\ref{claimB:Di-RM} for $i=1,2$, and $3$ implies that $D$ is complete.
Therefore,   we can assume  $|V(D)|\ge 5$. 

\begin{claimB}
\label{claimB:Di-complete}
If  $D_i$ is  a complete graph for some $i\in\{1,\ldots, l\}$, then $D$ is a complete graph.
\end{claimB}

\begin{cProofB}
Assume that $D_i$ is a complete graph for some $i\in\{1,\ldots, l\}$. It is easy to see that for any $j\in\{1,\ldots, l\}$ the graph $D_i-d_j$ contains a triangle. Since 
$D_i - d_j$ is contained in $D_j$, we get that $D_j$ is a complete graph for each $j\in\{1,\ldots, l\}$ by Claim~B\ref{claimB:Di-RM}. The proof of the claim is concluded by observing that for each pair of vertices $d$ and $d'$ of $D$ there is some $m\in \{1,\ldots, l\}$ such that both $d$ and $d'$ are contained in $D_m$.
\hfill $\blacksquare$
\end{cProofB}

\begin{claimB}
\label{claimB:D=Kn+1n}
If there is a pair of integers $i$ and $j$ from $\{1,\ldots, l\}$ such that
both $D_i$ and $D_j$ are isomorphic with $K_{n,n}$ for some $n$, then $d_i$ and $d_j$ are not adjacent.
\end{claimB}

\begin{cProofB} 
Let $m$ be an integer from $\{1,\ldots, l\}-\{i,j\}$ and note that $D_m$ is connected randomly matchable by Claim~B\ref{claimB:Di-RM}.
Observe that $D_m$ is $K_{n,n}$, since otherwise $D_m - d_i \subseteq D_i$ would contain a triangle. 
Since $D_j$ is $K_{n,n}$, the graph $D_j-d_i$ is $K_{n,n-1}$.
Let $A$ denote the set of vertices of $D$ lying in the larger partite set of $D_j$. By comparing $D_j$ and $D_j-d_i$ it is easy to see that 
 $d_i$ is adjacent to all vertices of $A$.
Furthermore, $D_i  = (D_j - d_i) \cup \{d_j\}$ and thus also $d_j$ is adjacent to all vertices of $A$. 
It follows that  both $d_i$ and $d_j$ are 
in $D_m$ adjacent to all vertices of $A\cap D_m$.  
The fact that $|V(D)|\ge 5$ implies  $n\ge 2$ and thus $A\cap D_m$ contains a vertex $d$. 
The proof is concluded by observing that 
$d_i$ and $d_j$ are not adjacent, since otherwise $D_m$ would contain the triangle $\{d,d_i,d_j\}$.
\hfill $\blacksquare$
\end{cProofB}

Recall that $l\ge 3$ and observe that if one of  $D_1, D_2,$ and $D_3$ is a complete graph, then we are done by Claim~B\ref{claimB:Di-complete}. Therefore, we can assume that $D_1,D_2,$ and $D_3$ are $K_{n,n}$ for some integer $n$. 
Let $M' = (M-\{s_1x_1,s_2x_2,s_3x_3 \})\cup \{s_1d_1,s_2d_2,s_3d_3\}$, and let $H'$ be an odd component of 
    $C'\cup\{x_1, x_2, x_3\}$.
Clearly, $G, M'$, and $H'$ satisfy the assumptions of Lemma~\ref{lem:matchingCut}, which in turn implies that $G- (H'\cup V(M')) = D-\{d_1,d_2,d_3\}$ is
connected randomly matchable.  Since $D-d_1$ is $K_{n,n}$ for some $n$ by our assumption, Observation~\ref{observation:Knn-two-vertices-removal} implies that $d_2$ and $d_3$ are adjacent. On the other hand, Claim~B\ref{claimB:D=Kn+1n} yields that $d_2$ and $d_3$  are not adjacent, which is a contradiction. The proof is now complete.
\end{proof}

We conclude this section by showing that the requirement on the number of vertices in Lemma~\ref{lemma:two-components} cannot be relaxed.
More precisely,  for every $k\ge 3$ we construct a $k$-connected equimatchable factor-critical graphs with $2k+1$ vertices and a $k$-cut $S$ such that $G-S$ has $k$ components and show that this bound is tight.

\begin{prop}
\label{prop:many-components}
Let $G$ be a $k$-connected equimatchable factor-critical graph with a $k$-cut $S$ for any $k\ge 3$. Then $G-S$ has at most $k$ components and this bound is tight for every $k\ge 3$.
\end{prop}

\begin{proof}
If $|V(G)|\ge 2k+3$, then $G-S$ has exactly $2\le k$ components by Lemma~\ref{lemma:two-components}. Therefore, we can assume that $|V(G)|\le 2k+1$. Clearly, the number of components of $G-S$ is at most $|V(G-S)| \le k+1$, with equality if and only if  $G-S$ consists from $k+1$ singletons. However, it is easy to see that if $G-S$ consists from $k+1$ singletons, then for arbitrary vertex $s$ of $S$ the graph  $G-s$ cannot have a perfect matching, which  contradicts  factor-criticality of $G$. 
Therefore, the number of components of $G-S$ is at most $k$.

To show that this bound is tight, for each $k\ge 3$ we construct a $k$-connected equimatchable factor-critical graph $G_k$ with $2k+1$ vertices and a $k$-cut $S$ such that $G-S$ has exactly $k$ components.
Let $V(G_k) = C\cup D\cup S$, where $C$ and $S$ form independent sets of $G_k$ with sizes $k-1$ and $k$, respectively, and $D$ is a copy of $K_2$. The edges of $G_k$ are precisely all the edges between $S$ and $G_k-S$ and the edge in $D$.
Clearly, the graph $G_k$ is $k$-connected for every $k\ge 3$. To show that $G_k -v$  has a perfect matching for each vertex $v$,  we distinguish whether $v$ belongs to $S$, $C$, or $D$.
 If $v$ is a vertex of $S$, then a perfect matching of $G_k-v$ can be constructed by taking the edge from $D$ and a matching between $S-v$ and $C$ covering all vertices of $(C\cup S) - v$.
 If $v$ is a vertex of $C$ or $D$, then 
$G_k-v$ contains $K_{k,k}$ as a subgraph and hence also admits a perfect matching.

To prove that any matching $M$ of $G_k$ can be extended to a maximum matching, we distinguish two cases according to whether $M$ contains the edge of $C$ or not. If $M$ contains the edge $c_1c_2$ of $C$, then
$M- \{c_1c_2\}$ is a matching in $G_k- \{c_1,c_2\}$, which in turn is isomorphic to $K_{k,k-1}$. Since $K_{k,k-1}$ is equimatchable, the matching $M-\{c_1c_2\}$ can be extended to a matching  $M$ of $G_k- \{c_1,c_2\}$ covering all but one vertex. It follows that $M'\cup \{c_1c_2\}$ is the desired maximum matching of $G_k$ containing $M$.
If $M$ does not contain the edge of $C$, then $M$ contains only edges from $G_k - E(C)$, which is isomorphic to an equimatchable graph $K_{k+1,k}$. It follows that $M$ can be extended to a matching of $G_k$ covering all but one vertex, which completes the proof.
\end{proof}

We note that Theorem~\ref{thm:Cgek-Dge3} and  \ref{thm:kCutEFC} cannot be extended to  graphs with connectivity $2$. More precisely, for graphs with connectivity 2 neither the fact that $G-S$ has two components with at least three vertices implies that the components are complete, nor  presence, respectively  absence, of an edge in $S$ forces the structure described in Theorem~\ref{thm:kCutEFC}. 



\section{Graphs with independence number $2$}

In this section we investigate the relationship between equimatchability and independence number. We focus on odd $k$-connected graphs with $k\ge 4$, at least $2k+3$ vertices, and a $k$-cut which separates at least two components with at least $3$ vertices and show that such graphs are equimatchable factor-critical if and only if their independence number equals $2$. In one direction, we show that if a graph with independence number $2$ is odd, then it is equimatchable, and if it is even, then it is very close to being equimatchable. In the reverse direction, we use the characterisation of $k$-connected equimatchable factor-critical graphs with at least $2k+3$ vertices and a $k$-cut separating at least two components with at least three vertices from Theorem \ref{thm:Cgek-Dge3} to show that if $k\ge 4$, then all such graphs have independence number $2$. Finally, we provide examples showing that it is not possible to extend these results to graphs in which every minimum cut separates a component with at most $2$ vertices --  even if such graphs are equimatchable factor-critical, they can have arbitrarily large independence number. Note that Proposition~\ref{prop:many-components} from the previous section shows that these result can neither be extended to graphs with at most $2k+1$ vertices, since in such graphs $G-S$ can have $k$ components and hence also independence number at least $k$.

We start with two propositions showing close relationship between equimatchable and almost-equimatchable graphs, and graphs with independence number $2$. In the following proof we assume that the reader is familiar with the concept of Gallai-Edmonds decomposition, see \cite{LP} for details. We use notation consistent with \cite{LP}, more precisely,
$D$ is the set of vertices of $G$ uncovered by at least one maximum matching of $G$. Furthermore, $A$ is the set of vertices of $G-D$ adjacent to at least one vertex of $D$ and $C$ is the set $V(G)-(A\cup D)$. For 
a discussion concerning how Gallai-Edmonds decomposition relates to equimatchable graphs see \cite{LPP}.

\begin{prop}
\label{prop:in-2-equimatchable}
Let $G$ be a graph with independence number $2$. If  $G$ is odd, then $G$ is equimatchable. If  $G$ is even, then
either $G$ is randomly matchable, or
$G$ is not equimatchable, has a perfect matching, and every maximal matching of $G$ leaves uncovered at most two vertices.
\end{prop}

\begin{proof}
For any maximal matching $M$ the set of vertices not covered by $M$ induces an independent set. Hence any maximal matching of a graph with independence number $2$ leaves uncovered at most $2$ vertices. Since the parity of the number of vertices not covered by a matching is the same as the parity of the number of vertices of the graph, if $G$ is odd, then any maximal matching of $G$ leaves uncovered exactly one vertex. Consequently, all maximal matchings of $G$ have the same size and $G$ is equimatchable. If $G$ is even, then every maximal matching of $G$ leaves uncovered $0$ or $2$ vertices.
We distinguish two cases: either $G$ is equimatchable, or not. If $G$ is equimatchable with a perfect matching, then it is isomorphic with $K_{2n}$ or $K_{n,n}$ for some nonnegative integer $n$ by \cite{sumner}. Suppose that $G$ is equimatchable and every maximal matching of $G$ leaves uncovered exactly $2$ vertices, our aim is to show that there are no such graphs. Since $G$ is even, it cannot be factor-critical. Furthermore, $G$ does not have a perfect matching and thus it has a nontrivial Gallai-Edmonds decomposition. It is well known that the number of vertices uncovered by any maximum matching equals the difference between the number of components in $D$ and the number of vertices in $A$, see for example \cite{LP}. It follows that there are at least $3$ components in $D$, which contradicts the fact that the independence number of $G$ is $2$.
The only remaining possibility is that $G$ is not equimatchable, in which case $G$ has a perfect matching, 
every its maximal matching leaves uncovered at most $2$ vertices, and it has a maximal matching which leaves uncovered precisely $2$ vertices.
\end{proof}

Odd graphs with independence number $2$ are described by the following proposition.

\begin{prop}
\label{prop:oddIS2}
Let $G$ be a connected odd  graph with independence number $2$. Then $G$ is either
factor-critical, or an union of two complete graphs, one even and one odd, joined by a set of pairwise incident edges.
\end{prop}

\begin{proof}
If $G$ is $2$-connected, then by \cite{LPP} it is either bipartite, or factor-critical. 
If $G$ is factor-critical, then there is nothing to prove. Therefore, we can assume that $G$ is bipartite.
Since each partite set of a bipartite graph form an independent set, each partite sets of $G$ has size at most $2$. From the fact that $G$ is odd follows that $G=P_3$ and thus it has a cutvertex. Therefore, there is no $2$-connected odd bipartite graph with independence number $2$, which completes the proof of the first case.

If $G$ has a cutvertex $v$, then $G-v$ has exactly two components, otherwise the independence number of $G$ would be at least $3$. Moreover, both components of $G-v$ are complete, since otherwise there would be an independence set with size $3$ consisting from two nonadjacent vertices of one component and any vertex of the second component. If there are two vertices $u$ and $w$ from different components of $G-v$ such that $v$ is not adjacent to neither of them, then again $\{u,v,w\}$ is an independent set of size $3$. Hence $v$ is adjacent with every vertex of at least one component of $G-v$ and $G$ is an union of two complete graphs, one even and one odd, joined by a set of pairwise incident edges.
\end{proof}

We now turn our attention to the independence number of equimatchable factor-critical graphs.

\begin{lemma}\label{lemma:matching-of-S}
Let $G$ be a $k$-connected equimatchable factor-critical graph with a $k$-cut $S$ for some $k\ge 3$. Assume that $G-S$ has precisely two components $C$ and $D$, both of them complete. Then for each vertex $s$ of $S$ there
is a matching $M$ containing only edges from $S- s$ such that $|S- V(M)| = 2$ if $k$ is even and $|S- V(M)| = 3$ if $k$ is odd.
\end{lemma}

\begin{proof} 
Let $s$ be a given vertex of $S$ and first assume that $k$ is even. One of the components of $G- S$, say $C$, is odd and the other is even. Let $c$ be a vertex of $C$ adjacent to $s$ and denote by $M_D$ and $M_C$ a perfect matching of $D$ and of $C- c$, respectively.
 Clearly, the set $N$ defined by $N=M_D\cup M_C\cup \{sc\}$ is a matching of $G$ and hence it can be extended to a matching $N'$ leaving only one vertex of $G$ uncovered. 
The only vertex not covered by $N'$ lies in $S$, 
therefore $N'- N$ is the desired matching.

In the rest of the proof we assume that $k$ is odd, which implies that  $|C|$ and $|D|$ have the same parity. First we consider the case where both $|C|$ and $|D|$ are even.
Let $M_C$ and $M_D$ be perfect matchings of $C$ and $D$, respectively. 
Since $G$ is equimatchable and factor-critical, the matching $M_C\cup M_D$ can be extended to a matching $N$ leaving only one vertex $s'$ of $G$ uncovered. Note that necessarily $s'$ lies in $S$.
Let $M = N- M_C\cup M_D$. If $s=s'$, then let $e$ be an arbitrary edge of $M$, otherwise let $e$ be the edge of $M$ incident with $s$. It is easy to see that $M- e$ is the desired matching.
\\
Finally we consider the case where both $|C|$ and $|D|$ are odd.
Let $s'$ be a vertex of $S$ different from $s$ and let 
$c$ be a vertex of $C$ and $d$ a vertex of $D$ adjacent to $s$ and $s'$, respectively. 
Furthermore, let $M_C$ and $M_D$ be perfect matchings of $C-c$ and $D-d$, respectively. 
Since $G$ is equimatchable and factor-critical, the
matching $N$ defined by $N=M_C\cup M_D\cup \{sc,s'd\}$ can be extended to a matching $N'$ leaving uncovered only
one vertex of  $G$. 
As in the previous cases, it is easy to see that the vertex uncovered by $N'$ lies in $S$. Therefore, 
 $N'-N$ is the desired matching, which completes the proof.
\end{proof}

\begin{lemma}
\label{lemma:unmatched-triple}
	Let $G$ be a $k$-connected  equimatchable factor-critical graph with at least $2k+3$ vertices and a $k$-cut $S$, where $k\ge 4$.
Assume that $G- S$ has two components $C$ and $D$, each with at least $3$ vertices.
Then 
for any vertices $s\in S$, $c\in C$,  and $d\in D$
the subgraph of $G$ induced by $\{c,d,s\}$ contains at least one edge.
\end{lemma}

\begin{proof}
Theorem~\ref{thm:Cgek-Dge3} implies that both $C$ and $D$ are complete and that $G-S$ does not have any other components. For the rest of the proof let $c,d$, and $s$ be arbitrary, but fixed, vertices of $G$ such that $c\in C$, $d\in D$, and $s\in S$.
We will need the following two claims.

\begin{claim}
\label{claim:matching-triple-even-components}
If there is a matching $M$  covering all vertices of $S-s$ such that $V(M)\cap \{c,d,s\}= \emptyset$ and both $C-V(M)$ and $D-V(M)$ are odd,  then the subgraph of $G$ induced by $\{c,d,s\}$ contains at least one edge.
\end{claim}

\begin{cProof}
Since both $C$ and $D$ are complete and both $C-V(M)$ and $D-V(M)$ are odd, the subgraphs of $G$ induced by $C-V(M)$ and $D-V(M)$ are odd complete graphs. Therefore, there are matchings $M_C$ and $M_D$ of $C- V(M)$ and $D-V(M)$ covering all vertices of $C-(V(M)\cup\{c\})$ and $D- (V(M)\cup\{ d\})$, respectively. It follows that $M' = M\cup M_C\cup M_D$ is a matching of $G$ covering all vertices of $G$ except $c,d$, and $s$. Since $G$ is equimatchable and factor-critical, $M'$ can be extended to a maximum matching of $G$, that is, a matching covering all but one vertices of $G$. Consequently, the subgraph of $G$ induced by $\{c,d,s\}$ contains at least one edge, as claimed.
\hfill $\blacksquare$
\end{cProof}

\begin{claim}
\label{lemmas:unmatched-triple-claim}
Let $s$ be a vertex from $S$.
If $s$ is adjacent to only one vertex of some component of $G-S$, then $s$ is adjacent to all vertices of the other component of $G-S$. 
\end{claim}

\begin{cProof}
Assume that $s$ is adjacent to a single vertex of $D$, say $d$.
Let $R = (S \cup \{d\}) - s$ and note that $R$ is a $k$-cut of $G$ such that $G-R$ has two components, namely  $C\cup\{s\}$ and $D-d$. 
If $|D|\ge 4$, then both components of $G- R$ have at least $3$ vertices and hence both are complete by Theorem~\ref{thm:Cgek-Dge3}. In particular, $s$ is adjacent to every vertex of $C$, as claimed.
If $|D|=3$, then the components of $G- R$ have size $2$ and at least $k+1$, respectively. 
By Theorem \ref{thm:kCutEFC} either $C\cup \{s\}$ is a complete graph, or $R\cup C\cup \{s\}$ is a complete bipartite graph minus a matching. Since $C$ is a complete graph containing a triangle, the graph $R\cup C\cup \{s\}$ cannot be bipartite and the claim follows.  
\hfill $\blacksquare$
\end{cProof}

By Lemma~\ref{lemma:matching-of-S} there is a matching $M$ containing only edges from $S-s$ such that $|S-V(M)|\le 3$. 
Let  $C' = C- c$, $D' = D- d$, and  $S' = S- (V(M)\cup \{s\})$. We distinguish three cases depending on the parity of $C$ and $D$. 

First let $k$ be even. 
Since $k$ is even, one of the components of $G-{S}$ is even, say $C$, and the other is odd. 
By Lemma~\ref{lemma:matching-of-S} the set $S'$ contains only one vertex, denote it by $s_{1}$. 
If $s_{1}$ is adjacent to some vertex in $C'$, then we are done by Claim~\ref{claim:matching-triple-even-components}. Otherwise, $c$ is the only neighbour of $s$ in $C$ and by Claim~\ref{lemmas:unmatched-triple-claim} the vertex $s_{1}$ is adjacent to all vertices of $D$. Since $k\ge 4$, 
$M$ contains at least one edge, say $s_2s_3$.
By Lemma~\ref{lem:independentEdges} there are two independent edges between $\{s_{2}, s_{3}\}$ and $D$. At least one of the edges, say the one incident with $s_{2}$, does not have $d$ as an endvertex. Denote this edge by $e$. 
Using Lemma~\ref{lem:independentEdges} again yields that there are
two independent edges between $\{s_{1}, s_{3}\}$ and $C$. Clearly, one of these edges is $s_{1}c$ and hence there is an edge $f$ between $s_{3}$ and $C'$. Since $|D|\ge 3$ and $s_{1}$ is adjacent to all vertices of $D$, there is an edge $g$ between $D'$ and $s_{1}$ 
such that $e$ and $g$ are independent.
Applying Claim~\ref{claim:matching-triple-even-components} to the matching $(M-\{s_{2}s_{3}\})\cup\{e,f,g\}$ finishes the proof.

Assume that  $k$ is odd and  that both components of $G-S$ are even. Since $k$ is odd, by Lemma~\ref{lemma:matching-of-S} the set $S'$ contains precisely  two vertices, denote them by $s_{1}$ and $s_{2}$.
 By  Lemma~\ref{lem:independentEdges} there are two independent edges between $S'$ and $C$ and thus at least one of them, say $s_1c'$, is not incident with $c$. 
 If there exists an edge $s_2d'$, where $d\in D'$, then the matching $M\cup \{s_1c',s_2d'\}$ satisfy the assumptions of the Claim~\ref{claim:matching-triple-even-components} and we are done.
If there is no edge between $s_2$ and $D'$, then $d$ is the only vertex of $D$ adjacent to $s_2$. By Lemma~\ref{lem:independentEdges} there are two independent edges between $D$ and $\{s_1,s_2\}$, one of them is necessarily $s_2d$ and thus the other is $s_1d''$, where $d''\neq d$. Furthermore, by Claim~\ref{lemmas:unmatched-triple-claim} the vertex $s_2$ is adjacent to all vertices $C$ and thus there is an edge $s_2c''$, where $c''\neq c$. Applying Claim~\ref{claim:matching-triple-even-components} to the matching $M\cup\{s_1d'',s_2c''\}$ completes the proof of this case.

Finally, if $k$ is odd and also both components of $G-S$ are odd, then again $S'$ has two vertices by Lemma~\ref{lemma:matching-of-S}. 
First observe that if there is a matching $M'$ between $S'$ and $C'$ or between $S'$ and $D'$,  which covers both vertices of $S'$,
 then applying Claim~\ref{claim:matching-triple-even-components} to $M\cup M'$ yields the desired result.
We proceed to show that it is always possible to construct such a matching.
Let $s_1$ and $s_2$ be the vertices in $S'$.
 Since $k$ is odd, the matching $M$ contains at least one edge of $S$, say $s_3s_4$.
 By  Lemma~\ref{lem:independentEdges} there is a set of two independent edges $M_C$ between $S'$ and $C$ and a set of two independent edges between $S'$ and $D$. 
If $M_C$  does not cover either $c$ or $d$, then $M\cup M_C$ is the desired matching  and we are done, similarly for $M_D$. 
Therefore, we can assume that all matchings between $S'$ and $C$ covering $S'$ cover also $c$ and analogously all matchings between $S'$ and $D$ covering $S'$ cover also $d$.
In the rest of the proof we distinguish two cases.

First assume that one of the vertices of $S'$, say $s_1$, is in $C$ adjacent only to $c$. By Claim~\ref{lemmas:unmatched-triple-claim}  the vertex $s_1$ is adjacent to all vertices of $D$. It follows that $s_2$ is in $D$ adjacent only 
to $d$, since othwerwise there would be a set of two independent edges between $S'$ and $D'$. By Lemma~\ref{lem:independentEdges}
there are $3$ independent edges $s_1c_1,s_2c_2,s_3c'$  between $C$ and $\{s_1,s_2,s_3\}$. Since $s_1$ is in $C$ adjacent only with $c$,
we necessarily have $c_1 = c$.
Similarly, by Lemma~\ref{lem:independentEdges}
there are $3$ independent edges $s_1d_1,s_2d_2,s_4d'$ between $D$ and $\{s_1,s_2,s_4\}$. 
Again, since $s_2$ is in $D$ adjacent only to $d$, we have $d_2 = d$. 
It follows that $M' = (M-s_3s_4)\cup \{s_1d_1,s_2c_1,s_3c', s_4d'\}$ satisfy the assumptions of  Claim~\ref{claim:matching-triple-even-components}, which completes the proof of this case.

Second assume that both vertices of $S'$ are adjacent to at least two vertices of both $C$ and $D$. It follows that there is a vertex $c'$ of $C'$ such that there is no edge between $C'-c'$ and $S'$, since otherwise there would be a set of two independent edges between $S'$ and $C'$. Similarly, there is a vertex $d'$ of $D'$ such that there is no edge between $D'-d'$ and $S'$. 
By Lemma~\ref{lem:independentEdges} there are three independent edges between $\{s_1,s_2,s_3\}$ and $C$.  
One of these edges is $s_3c''$, where $c''$ is different from both $c$ and $c'$,  since $s_1$ and $s_2$ are adjacent to precisely  two vertices of $C$. 
Without loss of generality we can assume that  the other two are $s_1c$ and $s_2c'$. 
Similarly, by Lemma~\ref{lem:independentEdges} there are three independent edges between $\{s_1,s_2,s_4\}$ and $D$. Again, one of these edges is $s_4d''$. Since both $s_1$ and $s_2$ are adjacent to precisely two vertices of $C$, each of them is adjacent to both $d$ and $d'$ and thus $s_1d'$ is an edge of $G$.
To  conclude the proof of  it suffices  to observe that the matching $(M-s_3s_4)\cup \{s_1d',s_2c',s_3c'', s_4d''\}$   satisfy the assumptions of  Claim~\ref{claim:matching-triple-even-components}. 
\end{proof}

\begin{lemma}
\label{lemma:efc-in-2}
Let $G$ be a $k$-connected  equimatchable factor-critical graph with at least $2k+3$ vertices and a $k$-cut $S$ such that $G-S$ has two components with at least $3$ vertices, where $k\ge 4$.
Then the independence number of $G$ is $2$.
\end{lemma}

\begin{proof}
By Theorem~\ref{thm:Cgek-Dge3} both $C$ and $D$ are complete and $G-S$ does not have any other components.
 Since both $C$ and $D$ are complete, no independent set of $G$ can contain more than one vertex from any of them. Observe that $G$ cannot have an independent set $\{c,d,s\}$ where $c\in C$, $d\in D$, and $s\in S$ by Lemma~\ref{lemma:unmatched-triple}.

There are two remaining possible types of independent sets of size $3$ in $G$. More precisely,  either $G$ has an independent set consisting of  3 vertices of $S$, or $G$ has an independent set consisting of  2 vertices of $S$ and a vertex of $C\cup D$. For a contradiction suppose that $I$ is such an independent set of size 3 in $G$.
Let $T= I\cap S$ and let $C' = C-I$.
If $C'$ is odd, let $F$ be the set containing an arbitrary edge between $C'$ and $S- T$, otherwise let $F=\emptyset$. Furthermore, if $C'$ is odd, let $T' = T\cup\{s\}$, where
$s$ is the vertex of $S$ incident with the edge in $F$, otherwise let $T' = T$.
It is not difficult to see that if $D$ is odd, then there is a vertex $s'$ in $S- T'$. Therefore, if $D$ is odd, then by Lemma~\ref{lem:independentEdges} there is a vertex $d$ of $D$ adjacent to $s'$ and we set $F'=F\cup \{ds'\}$. If $D$ is even, we set $F'=F$. Let $U$ be the set of vertices covered by the edges in $F'$.
Since $|C\cup D|\ge k-1$ by the assumptions,  Lemma \ref{lem:independentEdges} implies that there is a matching $M$ between $S-(T\cup U)$ and $(C'\cup D) - U$ such that both $C'-V(M\cup F')$ and $D-V(M\cup F')$ are even and $M$ covers all vertices of $S-(T\cup U)$. 
 Finally, let $M_C$ and $M_D$ be  perfect matchings of $C'- (U\cup V(M))$ and $D- (U\cup V(M))$, respectively. It is easy to see that the matching $M'$ defined by $M'=M\cup M_C\cup M_D\cup F'$ covers all vertices of $G$ except $I$. Since $G$ is equimatchable and factor-critical, $M'$ can be extended to a matching leaving uncovered exactly one vertex. It follows that $I$  contains an edge,  contradicting
the fact that it is an independent set.
The proof is now complete.
\end{proof}

The following theorem is the main result of this section.

\begin{theorem}
\label{thm:efc-iff}
Let $G$ be a $k$-connected odd graph with at least $2k+3$ vertices
and a $k$-cut $S$ such that $G- S$ has two components with at least $3$ vertices, where $k\ge 4$. Then $G$ 
has  independence number at most $2$ if and only if 
it is equimatchable and factor-critical.
\end{theorem}

\begin{proof}
If $G$ is equimatchable and factor-critical, then its independence number is $2$ by Lemma~\ref{lemma:efc-in-2}.
\\
In the reverse direction, assume that $G$ has independence number $2$. Then $G$ is equimatchable by Proposition~\ref{prop:in-2-equimatchable}.
Furthermore, Proposition~\ref{prop:oddIS2} implies that $G$ is factor-critical, which completes the proof.
\end{proof}

The following theorem reveals further connection between independence number 2 and equimatchable graphs.

\begin{theorem}
A $2$-connected odd graph  $G$ with at least 4 vertices has independence number at most 2 if and only if $G$ is equimatchable and factor-critical and $G\cup\{e\}$ is equimatchable for each edge of the complement $\overline{G}$ of $G$.
\end{theorem}

\begin{proof}
If $G$ has independence number $2$, then it is equimatchable and factor-critical by Proposition~\ref{prop:in-2-equimatchable} and \ref{prop:oddIS2}.
Clearly, for any edge $e$ of the complement $\overline{G}$ of $G$ the graph $G\cup \{e\}$ has again independence number 2 and thus it is equimatchable by Proposition~\ref{prop:in-2-equimatchable}.
In the reverse direction, assume that $G$ is equimatchable  factor-critical and that $G\cup\{e\}$ is equimatchable for every edge of the complement $\overline{G}$ of $G$. 
For a contradiction suppose that $G$ has an independent set $\{x,y,z\}$ of size $3$. Since $G$ is factor-critical, $G-z$ has a perfect matching $M$. Let $x'$ and $y'$ be the vertices matched with $x$, respectively $y$, by $M$. 
Mote that since $\{x,y,z\}$ is an independent set we have $x\neq x'$ and $y\neq y'$.
It is easy to see that $M' = (M-\{xx',yy'\})\cup \{x'y'\}$ is a maximal matching of $G\cup\{x'y'\}$ which leaves uncovered precisely three vertices. 
On the other hand, $M$ is a matching of $G\cup \{x'y'\}$  leaving uncovered precisely one vertex. 
It follows that $M'$ is a maximal matching of $G\cup\{x'y'\}$, which is not maximum, contradicting equimatchability of $G\cup \{x'y'\}$. We conclude that the independence number of $G$ is at most 2, which completes the proof.
\end{proof}

Equimatchable graphs $G$ such that $G\cup\{e\}$ is equimatchable for every edge $e$ of the complement $\overline{G}$ of $G$ are further investigated in \cite{EK:extremal}, together with other extremal classes of equimatchable graphs.

\medskip
Our final two results show that Lemma~\ref{lemma:efc-in-2}, and thus also Theorem~\ref{thm:efc-iff}, can  be extended neither to equimatchable graphs without two components with at least $3$ vertices, nor to the case of graphs with connectivity $3$.

\begin{prop}
For every triple of integers $n,k,$ and $m$ such that $k\ge 3$ and $m\in\{1,2\}$ there is a $k$-connected equimatchable factor-critical graph $G$ with 
an independent set of size at least $n$ and 
a $k$-cut $S$ such that $G-S$ has a component of size $m$.
\end{prop}

\begin{proof}
First assume that $m=1$.
Let $l = \max\{n,k\}$ and 
denote by $H$ 
a copy of $K_{l,l}$. Choose a set $S$ of $k$ vertices of $H$ in such a way that $S$ contains at least one vertex from each partite set of $H$. The desired graph $G$ is constructed by taking a new vertex $v$ and joining it with every vertex in $S$.
Clearly, $G$ is $k$-connected and $S$ is a $k$-cut of $G$.
Since $\{v\}$ is a component of $G-S$ and $m=1$, the graph $G-S$ has a component with $m$ vertices.
Furthermore, it is easy to directly verify that $G$ is factor-critical and equimatchable.
The proof of this case is concluded by observing that each partite set of $H$ forms an 
independent set of $G$ with size $l\ge n$.

Now we assume that $m=2$.
Let $l = \max\{n,k\}$ and 
denote by $H_1$ 
a copy of $K_{l,l+1}$ and by $H_2$ a copy of $K_2$.
Denote by $S$ a set of $k$ vertices from the larger partite set of $H_1$. 
The desired graph $G$ is constructed by joining both vertices of $H_2$ with all vertices of $S$.
It can be easily verified that the resulting graph  is $k$-connected, equimatchable, and factor-critical.
Clearly $S$ is a $k$-cut of $G$ such that $G-S$ has a component with $m$ vertices. Finally,  $G$ contains an independent set with $l+1 \ge n$ vertices, which completes the proof.
\end{proof}

\begin{prop}
For every pair of odd integers $m$ and $n$ such that $m+n\ge 4$ there is a $3$-connected equimatchable factor-critical graph $G$ with independence number $3$ and a $3$-cut $S$ such that $G-S$ has two components with sizes $m$ and $n$, respectively.
\end{prop}

\begin{proof}
For any given pair of positive odd integers  $m$ and $n$ we  construct a graph $G(m,n)$ with the required properties as follows. 
Let $C$, respectively  $D$, be a copy of the complete graph on $m$ and $n$ vertices, respectively and let $S$ be an independent set on $3$ vertices. To obtain $G(m,n)$ we join every vertex of $C\cup D$ with every vertex of $S$. Since $m+n\ge 4$, the graph $G(m,n)$ is $3$-connected. 

To prove that $G(m,n)$ is factor-critical, first let $v\in(C\cup D)$ and let $G'=G(m,n)-v$. It is easy to see that  there is a set $M$ of  $3$ independent edges  between $S$ and $(C\cup D)-v$ such that  $G'-V(M)$ consists of two even complete graphs. Therefore, $M$ can be extended to a perfect matching of $G'$. 
If $v\in S$, then there is an edge $sc$ between $S-v$ and $C$ and an edge $s'd$, independent from $sc$, between $S-v$ and $D$. Since removing $\{s,s',c,d \}$ from $G-v$ yields two even complete  components,
$\{sc,s'd\}$ can be extended to a perfect matching of $G-v$, which in turn is factor-critical, as claimed.

In the rest of the proof we show that $G(m,n)$ is equimatchable.
Let $M$ be a maximal matching of $G(m,n)$.
Since $C$ and $D$ are complete, $M$ leaves uncovered at most one vertex of each $C$ and $D$.
Assume that  $M$ leaves uncovered a vertex $c$ from $C$. Since  $M$ is maximal, then clearly $M$ must cover all vertices of $S$. Therefore, $D-V(M)$ is an even complete graph and thus $M$ covers all vertices of $D$ by its maximality. We conclude that $c$ is the only vertex of $G(m,n)$ uncovered by $M$ and hence $M$ is a maximum matching. Analogous argument also shows that if $M$ leaves uncovered a vertex from $D$, then $M$ is maximum.

Therefore, we can assume that $M$ covers all vertices from $C\cup D$. Since both $C$ and $D$ are odd, to cover all vertices of $C\cup D$ the matching $M$ has to cover  precisely two vertices of $S$. Consequently, $M$ leaves uncovered exactly one vertex of $S$ and $M$ is maximum, as required.
Since both $C$ and $D$ are odd,  if $M$  does not leave uncovered exactly one vertex of $S$ leaves uncovered also at least one vertex of $C\cup D$, and hence $M$ cannot be a maximal matching of $G(m,n)$.
\end{proof}

\subsection*{Acknowledgement}
Research reported in this paper was partially supported by 
Ministry of Education, Youth, and Sport of Czech Republic, Project No.~CZ.1.07/2.3.00/30.0009.


\small
\begin{thebibliography}{99}


\bibitem{CHS:2012}
M.~Chudnovsky and P.~Seymour,
\newblock Packing seagulls.
\newblock {\em Combinatorica}, 32 (2012), 251--282.

\bibitem{EK:2013}
E.~Eiben and M.~Kotrb\v{c}\'ik, 
\newblock Equimatchable graphs on surfaces.
\newblock J. Graph Theory in press, DOI: 10.1002/jgt.21859

\bibitem{EK:extremal}
E.~Eiben and M.~Kotrb\v{c}\'ik, 
\newblock Extremal equimatchable graphs.
\newblock In preparation. 

\bibitem{favaron:1986}
O.~Favaron,
\newblock Equimatchable factor-critical graphs.
\newblock {\em J. Graph Theory}, 10 (1986), 439--448.


\bibitem{grunbaum}
B.~Gr\"unbaum,
\newblock Matchings in polytopal graphs.
\newblock {\em Networks}, 4 (1974), 175--190.


\bibitem{KPS}
K.~Kawarabayashi, M.~D.~Plummer, and A.~Saito,
\newblock On two equimatchable graph classes.
\newblock {\em Discrete Math.}, 266 (2003), 263--274.

\bibitem{konig}
D.~K\"onig, \newblock Graphok es matrixok. \newblock {\em Math. Fiz. Lapok} 38 (1931), 116--119. (In Hungarian.)

\bibitem{LPP}
M.~Lesk, M.~D. Plummer, and W.~R. Pulleyblank,
\newblock Equi-matchable graphs,
\newblock {\em Graph Theory and Combinatorics,
  Proc. Cambridge Combinatorial Conference in Honour of Paul Erdos}.
 B.~Bollob{\'a}s, (editor),  
Academic Press, London (1984), pp. 239--254. 

\bibitem{lewin}
M.~Lewin,
\newblock Matching-perfect and cover-perfect graphs.
\newblock {\em Israel J. Math.}, 18 (1974), 345--347.


\bibitem{LP}
L.~Lov\'asz and M.~D. Plummer,
\newblock {\em Matching Theory}.
\newblock North-Holland, Amsterdam, 1986.

\bibitem{meng}
D.~H.-C. Meng,
\newblock {\em Matchings and coverings for graphs}.
\newblock PhD thesis, Michigan State University, East Lansing, MI, 1974.

\bibitem{plummer:1994}
M.~D.~Plummer,
\newblock Extending matchings in graphs: A survey.
\newblock  {\em Discrete Math.}, 127 (1994), 277--292.

\bibitem{plummer:2007}
M.~D.~Plummer,
\newblock Graph factors and factorization: 1985--2003: A survey.
\newblock  {\em Discrete Math.}, 307 (2007), 791--821.

\bibitem{plummer:2008}
M.~D.~Plummer,
\newblock Recent Progress in Matching Extension.
\newblock  {\em Building Bridges.} M. Gr\"otchel et al. (eds), Springer, Berlin Heidelberg (2008), pp. 427--454.

\bibitem{PST:2003}
M.~D.~Plummer, M.~Stiebitz, and  B.~Toft,
\newblock On a special case of Hadwiger's conjecture.
\newblock  {\em Discuss. Math. Graph Theory}, 23 (2003), 333--363.



\bibitem{sumner}
D.~P.~Sumner,
\newblock Randomly matchable graphs.
\newblock {\em J. Graph Theory}, 3 (1979), 183--186.


\end{thebibliography}
\end{document}